\documentclass[%12pt,
reqno]{amsart}%{article}
\usepackage{latexsym,amsmath,amssymb,amscd}
\usepackage[all]{xy}
\usepackage{enumerate}
\usepackage{hyperref}

\def\today{\ifcase \month \or
   January \or February \or March \or April \or
   May \or June \or July \or August \or
   September \or October \or November \or December \fi
   \space\number\day , \number\year}

\makeatletter
  \newcommand\@dotsep{4.5}
  \def\@tocline#1#2#3#4#5#6#7{\relax
     \ifnum #1>\c@tocdepth % then omit
     \else
     \par \addpenalty\@secpenalty\addvspace{#2}%
     \begingroup \hyphenpenalty\@M
     \@ifempty{#4}{%
     \@tempdima\csname r@tocindent\number#1\endcsname\relax
        }{%
         \@tempdima#4\relax
           }%
      \parindent\z@ \leftskip#3\relax \advance\leftskip\@tempdima\relax
      \rightskip\@pnumwidth plus1em \parfillskip-\@pnumwidth
       #5\leavevmode\hskip-\@tempdima #6\relax
       \leaders\hbox{$\m@th
       \mkern \@dotsep mu\hbox{.}\mkern \@dotsep mu$}\hfill
       \hbox to\@pnumwidth{\@tocpagenum{#7}}\par
       \nobreak
        \endgroup
         \fi}
\makeatother
\begin{document}

\makeatletter
\@addtoreset{figure}{section}
\def\thefigure{\thesection.\@arabic\c@figure}
\def\fps@figure{h,t}
\@addtoreset{table}{bsection}

\def\thetable{\thesection.\@arabic\c@table}
\def\fps@table{h, t}
\@addtoreset{equation}{section}
\def\theequation{%\thesection.
\arabic{equation}}
\makeatother

\newcommand{\bfi}{\bfseries\itshape}
\newtheorem{theorem}{Theorem}
\newtheorem{corollary}[theorem]{Corollary}
\newtheorem{criterion}[theorem]{Criterion}
\newtheorem{definition}[theorem]{Definition}
\newtheorem{example}[theorem]{Example}
\newtheorem{lemma}[theorem]{Lemma}
\newtheorem{notation}[theorem]{Notation}
\newtheorem{problem}[theorem]{Problem}
\newtheorem{proposition}[theorem]{Proposition}
\newtheorem{remark}[theorem]{Remark}
\numberwithin{theorem}{section}
\numberwithin{equation}{section}

%%% Todo
\newcommand{\todo}[1]{\vspace{5 mm}\par \noindent
\framebox{\begin{minipage}[c]{0.85 \textwidth}
\tt #1 \end{minipage}}\vspace{5 mm}\par}
%%%

\renewcommand{\1}{{\bf 1}}

\newcommand{\hotimes}{\widehat\otimes}

\newcommand{\Ad}{{\rm Ad}}
\newcommand{\ad}{{\rm ad}}
\newcommand{\Alt}{{\rm Alt}\,}
\newcommand{\Ci}{{\mathcal C}^\infty}
\newcommand{\comp}{\circ}
\newcommand{\wt}{\widetilde}

\newcommand{\ph}{\text{\bf P}}
\newcommand{\conv}{{\rm conv}}
\newcommand{\de}{{\rm d}}
\newcommand{\ee}{{\rm e}}
\newcommand{\ev}{{\rm ev}}
\newcommand{\fimes}{\mathop{\times}\limits}
\newcommand{\id}{{\rm id}}
\newcommand{\ie}{{\rm i}}
\newcommand{\End}{{\rm End}\,}
\newcommand{\Gr}{{\rm Gr}}
\newcommand{\GL}{{\rm GL}}
\newcommand{\Hilb}{{\bf Hilb}\,}
\newcommand{\Hom}{{\rm Hom}}
\renewcommand{\Im}{{\rm Im}}
\newcommand{\Ker}{{\rm Ker}\,}
\newcommand{\Lie}{\textbf{L}}
\newcommand{\lf}{{\rm l}}
\newcommand{\Loc}{{\rm Loc}\,}
\newcommand{\pr}{{\rm pr}}
\newcommand{\Ran}{{\rm Ran}\,}
\renewcommand{\Re}{{\rm Re}}
\newcommand{\supp}{{\rm supp}\,}

%amenability
\newcommand{\Cb}{{\mathcal C}_b}
\newcommand{\UCb}{{\mathcal U}{\mathcal C}_b}
\newcommand{\LUCb}{{\mathcal L}{\mathcal U}{\mathcal C}_b}
\newcommand{\RUCb}{{\mathcal R}{\mathcal U}{\mathcal C}_b}

\newcommand{\Tr}{{\rm Tr}\,}
\newcommand{\Tran}{\textbf{Trans}}

\newcommand{\CC}{{\mathbb C}}
\newcommand{\NN}{{\mathbb N}}
\newcommand{\RR}{{\mathbb R}}
\newcommand{\TT}{{\mathbb T}}
\newcommand{\ZZ}{{\mathbb Z}}

\newcommand{\G}{{\rm G}}
\newcommand{\U}{{\rm U}}
\newcommand{\Gl}{{\rm GL}}
\newcommand{\SL}{{\rm SL}}
\newcommand{\SU}{{\rm SU}}
\newcommand{\VB}{{\rm VB}}

\newcommand{\Ac}{{\mathcal A}}
\newcommand{\Bc}{{\mathcal B}}
\newcommand{\Cc}{{\mathcal C}}
\newcommand{\Dc}{{\mathcal D}}
\newcommand{\Ec}{{\mathcal E}}
\newcommand{\Fc}{{\mathcal F}}
\newcommand{\Gc}{{\mathcal G}}
\newcommand{\Hc}{{\mathcal H}}
\newcommand{\Kc}{{\mathcal K}}
\newcommand{\Nc}{{\mathcal N}}
\newcommand{\Oc}{{\mathcal O}}
\newcommand{\Pc}{{\mathcal P}}
\newcommand{\Qc}{{\mathcal Q}}
\newcommand{\Rc}{{\mathcal R}}
\newcommand{\Sc}{{\mathcal S}}
\newcommand{\Tc}{{\mathcal T}}
\newcommand{\Uc}{{\mathcal U}}
\newcommand{\Vc}{{\mathcal V}}
\newcommand{\Wc}{{\mathcal W}}
\newcommand{\Xc}{{\mathcal X}}
\newcommand{\Yc}{{\mathcal Y}}
\newcommand{\Zc}{{\mathcal Z}}
\newcommand{\Ag}{{\mathfrak A}}
\renewcommand{\gg}{{\mathfrak g}}
\newcommand{\hg}{{\mathfrak h}}
\newcommand{\mg}{{\mathfrak m}}
\newcommand{\nng}{{\mathfrak n}}
\newcommand{\pg}{{\mathfrak p}}
\newcommand{\Gg}{{\mathfrak g}}
\newcommand{\Lg}{{\mathfrak L}}
\newcommand{\Mg}{{\mathfrak M}}
\newcommand{\Sg}{{\mathfrak S}}
\newcommand{\Ug}{{\mathfrak u}}
\newcommand{\zg}{{\mathfrak z}}

\markboth{}{}

\makeatletter
\title
[Amenability and representation theory of pro-Lie groups]{Amenability and representation theory\\ of pro-Lie groups}
\author{Daniel Belti\c t\u a}
\address{Institute of Mathematics ``Simion Stoilow'' of the Romanian Academy,
P.O. Box 1-764, Bucharest, Romania}
\email{Daniel.Beltita@imar.ro, beltita@gmail.com}
\author{Amel Zergane}
\address{Higher Institute of Applied Sciences and Technology of Sousse, Mathematical Physics Laboratory, Special Functions and Applications,
City Ibn Khaldoun 4003, Sousse, Tunisia}
\email{amel.zergane@yahoo.fr}
%\thanks{}
\date{
%January 2, 2016.
\today.
%\textbf{File name}: \texttt{BZ1\_arxiv.tex}}
%\dedicatory{}
}
%\date{}
\makeatother

\begin{abstract}
We develop a semigroup approach to representation theory
for pro-Lie groups satisfying suitable amenability conditions.
As an application of our approach,
we establish a one-to-one correspondence between equivalence classes of unitary irreducible representations and
coadjoint orbits
for a class of pro-Lie groups
including all connected locally compact nilpotent groups and
 arbitrary infinite direct products of nilpotent Lie groups.
The usual $C^*$-algebraic approach to group representation theory
positivey breaks down for infinite direct products
of non-compact locally compact groups, hence the description of their unitary duals
in terms of coadjoint orbits is particularly important whenever it is available,
being the only description known so far. 
\end{abstract}

\maketitle

\tableofcontents

\section{Introduction}

Pro-Lie groups are topological groups that are isomorphic to projective limits of finite-dimensional Lie groups,
and their structure theory and Lie theory were developed to an impressive extent in the monograph \cite{HM07}.
Every connected locally compact group but also every infinite direct product of finite-dimensional Lie groups
belong to that remarkable class of topological groups.
The precise relation between pro-Lie groups and infinite-dimensional Lie groups was also investigated in \cite{HN09}.
Irreducible representation theory of pro-Lie groups was perhaps less developed so far,
so in the present paper we try to fill that gap, following the pattern
of the method of coadjoint orbits from representation theory of nilpotent Lie groups
(see Theorem~\ref{O4} below).
It is maybe unexpected that a substantial part of our investigation can be developed
on the level of semitopological semigroups,
and their amenability properties actually play a key role in our approach.

To describe our results in more detail,
let $G$ be a topological group with Lie algebra $\Lie(G)$ as in \cite{HM07}.
One has the coadjoint action $\Ad_G^*\colon G\times\Lie(G)^*\to\Lie(G)^*$,
with the corresponding set of coadjoint orbits denoted by $\Lie(G)^*/G$.
Also denote by $\widehat{G}$ the set of equivalence classes of unitary irreducible representations
of $G$.
In the spirit of the orbit method, there should be a (at least partially defined)
correspondence $\Psi_G\colon \widehat{G}\to \Lie(G)^*/G$.
Existence of $\Psi_G$ for pretty large classes of pro-Lie groups is established in Theorem~\ref{O4} below, 
and the coadjoint orbits in the image of $\Psi_G$ will be said to satisfy the \emph{integrality condition}.
The study of these orbits beyond the usual setting of Lie groups is a major motivation of the present paper.
And it turns out that the key to reducing the study of $\Psi_G$ from topological groups to Lie groups
is the amenability.
Additional motivation for us is the study of separation properties of unitary representations
in terms of suitable moment sets,
which we plan to extend from Lie groups (see \cite{ASZ11} and \cite{Ze11})
to more general topological groups, using the methods developed here and the moment sets introduced in \cite{BNi15}.

The above integrality terminology claims its origin in algebraic topology (cohomology with integer coefficients),
but for our purposes here it is more relevant to recall that if $G=\TT$ is the unit circle
regarded as a compact Lie group in the usual way, then its coadjoint orbits are the singleton sets $\{t\}$ with $t\in\RR$,
and the integrality condition on $\{t\}$ essentially means $t\in\ZZ$.
This is 
a special instance of the duality theory for locally compact abelian groups.
In fact, every topological abelian group has a Lie algebra, and its coadjoint orbits are the singleton sets $\{\xi\}$
with $\xi\in\Lie(G)^*$.
Some non-locally-compact abelian groups have only one unitary representation, namely the trivial representation
(see for instance \cite{Ba91}),
hence only one orbit satisfying the integrality condition, namely~$\{0\}$.
However, in the locally compact case, the Pontryagin duality clarifies
that there exist many coadjoint orbits that satisfy the integrality condition,
and they are parameterized by the dual group.
It is very interesting that although coadjoint orbits were not explicitly mentioned in connection with the duality theory of locally compact abelian groups, one of the most natural approaches to that theory is based on approximation with abelian Lie groups via projective limits
(see for instance \cite{Mr77}),
and the dual object of an abelian Lie group is directly related to coadjoint orbits.

We investigate the above ideas beyond the case of abelian groups,
in the spirit of the Kirillov theory on nilpotent Lie groups (see \cite{Ki62}).
We wish to study classes~$\Cc$ of topological groups $G$ for which there exists
a map $\Psi_G\colon\widehat{G}\to \Lie(G)^*/G$ that is equal to the Kirillov correspondence
if $G$ is a connected nilpotent Lie group and is functorial:
If $H$ is another topological group in the class $\Cc$ and $p\colon G\to H$ is any continuous surjective homomorphism,
then one has the commutative diagram
$$\begin{CD}
\widehat{H} @>{\widehat{p}}>> \widehat{G} \\
@V{\Psi_H}VV @VV{\Psi_G}V \\
\Lie(H)^*/H @>{\check p}>> \Lie(G)^*/G
\end{CD}$$
where $\widehat{p}$ and $\check p$ are the maps that are canonically associated to $p$. 
This abstract approach is applied here to
representation theory of pro-Lie groups,
thus investigating to what extent the method of coadjoint orbits carries over via projective limits
beyond the framework of Lie groups.
This study is strongly motivated by \cite[Ch. 9, Postscript]{HM07}.

The structure of this paper is as follows.
In Section~\ref{prelims} we collect some terminology, notation, and a couple of basic examples of pro-Lie groups
to which our methods are particularly well suited:
\begin{itemize}
\item Infinite direct products of connected nilpotent Lie groups.
\item Connected nilpotent locally compact groups.
\end{itemize}
In Section~\ref{factor} we develop our semigroup amenability approach to factor representations,
the main technical here being Theorem~\ref{am2a}.
We also establish here in Proposition~\ref{product}, for later use,
amenability of any infinite direct product of amenable topological groups,
a result that we were not able to locate in the existing literature,
although it has been long known that it fails to be true for discrete groups \cite[page 517]{Da57}.
In Section~\ref{main} we establish our main result on
the bijective correspondence between the unitary dual and the integral coadjoint orbits
of pro-Lie groups that satisfy suitable amenability conditions (Theorem~\ref{O4}).
We then specialize it for the aforementioned two classes of groups (Corollaries \ref{O4_cor1} and \ref{O4_cor2})
and we also discuss some specific examples.
As the connected nilpotent locally compact groups are projective limits of nilpotent Lie groups that may not be simply connected,
(see Examples \ref{solenoid} and \ref{sol} below), 
we also need a few results from representation theory of these Lie groups and we collected these
with full proofs in Appendix~\ref{appA}, because again they do not seem to be easily accessible in the earlier literature.

\section{Preliminaries}\label{prelims}

In this section we record some terminology and notation to be used throughout this paper.
We recall that a monoid is a semigroup with unit element,
and a submonoid is any subsemigroup that contains the unit element.
We denote by~$\1$ the unit element of any monoid whose composition law is denoted multiplicatively,
as well as the identity map of any vector space.
If $S$ is any semigroup, then a \emph{normal subsemigroup} of $S$ is any subset $N\subseteq S$
such that $NN\subseteq N$ and $sN=Ns$ for every $s\in S$.
In this case we can define an equivalence relation on $S$ by
$$s_1\sim s_2\iff s_1N=s_2N$$
and the corresponding quotient set $S/N:=\{sN\mid s\in S\}$ has the natural structure of a semigroup
with the operation $sN\cdot tN:=stN$.
It is straightforward to check that this operation on $S/N$ is well defined and the quotient map
$$p_N\colon S\to S/N,\quad s\mapsto sN$$
is a homomorphism of semigroups.
If $S$ is a monoid or a group, then so is $S/N$.

A homomorphism of semigroups $S_1\to S_2$ is said to be normal if its image is a normal subsemigroup of~$S_2$.

Now assume that the semigroup $S$ is equipped with a topology.
We say that $S$ is a {\it right} (respectively, {\it left})
{\it topological semigroup}
if for each $s\in S$ the mapping $S\to S$, $t\mapsto ts$
(respectively, $t\mapsto st$) is continuous.
Moreover $S$ is a {\it semitopological semigroup} if
it is both left and right topological.

\begin{notation}
\normalfont
For every complex Hilbert space $\Hc$ we denote by $\Bc(\Hc)$ its set of all continuous linear operators.
Besides the operator norm topology, we will use the strong operator topology and the strong$^*$ operator topology
on $\Bc(\Hc)$, defined as follows.
If $\{a_j\}_{j\in J}$ is a set in $\Bc(\Hc)$ and $a\in\Bc(\Hc)$, then
$\lim\limits_{j\in J}a_j=a$ in the strong operator topology if $\lim\limits_{j\in J}\Vert a_j\xi-a\xi\Vert=0$ for every
$\xi\in\Hc$,
while  $\lim\limits_{j\in J}a_j=a$ in the strong$^*$ operator topology if we have
both $\lim\limits_{j\in J}a_j=a$ and $\lim\limits_{j\in J}a_j^*=a^*$ in the strong operator topology.
The strong operator topology and the strong$^*$ operator topology coincide on
the unitary group
$$U(\Hc):=\{u\in\Bc(\Hc)\mid u^*u=uu^*=\1\}$$
and turn $U(\Hc)$ into a topological group.

The contraction monoid of $\Hc$ is defined as
$C(\Hc):=\{a\in\Bc(\Hc)\mid\Vert a\Vert\le 1\}$.
It is well known that the operator multiplication map $C(\Hc)\times C(\Hc)\to C(\Hc)$, $(a,b)\mapsto ab$,
is separately continuous with respect to any of the strong operator topology and the strong$^*$ operator topology,
hence any of these two topologies turn $C(\Hc)$ into a semitopological monoid.
\end{notation}

For any monoid $S$, a \emph{filter basis of submonoids} is a set $\Nc\ne\emptyset$ of submonoids
with the property that for every $N_1,N_2\in\Nc$ there exists $N_3\in\Nc$ with $N_3\subseteq N_1\cap N_2$.
If moreover $S$ is equipped with a topology,
then we say that $\Nc$ is \emph{converging to the identity} if for every neighborhood $V$ of $\1\in S$ there exists $N\in\Nc$
with $N\subseteq V$.

\begin{remark}\label{N0}
\normalfont
We say that the topological group $G$ is \emph{complete}
if every Cauchy net in $G$ is convergent, where a Cauchy net is any family $\{g_j\}_{j\in J}$
of elements of $G$, where $J$ is a directed set, such that for every neighborhood $V$ of $\1\in G$
there exists $i\in J$ such that $g_jg_k^{-1}\in V$ for all $j,k\in J$ with $j,k\ge i$.
Every locally compact group is complete by \cite[Rem. 1.31]{HM07}
and then so is every projective limit of Lie groups by \cite[Lemma 1.32(iii)]{HM07}.
We also recall from \cite[Th. 1.33]{HM07} that if
$G$ is a complete topological group with a filter basis $\Nc$ of closed normal subgroups
converging to the identity,
then the map
$$G\to \prod\limits_{N\in\Nc}G/N, \quad g\mapsto(p_N(g))_{N\in\Nc}$$
gives an isomorphism of topological groups $\gamma\colon G\to\varprojlim\limits_{N\in\Nc}G/N$.
 
In this setting, let us establish some further notation to be used throughout this paper, unless otherwise mentioned. 
Assume that
for every $N\in\Nc$ the topological group $G/N$ is a Lie group.
If $N_1,N_2\in \Nc$ with $N_1\subseteq N_2$, then the map
$$p_{N_1,N_2}\colon G/N_1\to G/N_2,\quad gN_1\mapsto gN_2$$
is a well-defined continuous homomorphism of Lie groups,
hence $p_{N_1,N_2}$ is automatically smooth.
Thus the family $\{G/N\}_{N\in \Nc}$ is organized as a projective system of Lie groups,
whose projective limit is isomorphic to the locally compact group $G$
(see \cite[Th. 1.30, 1.33]{HM07}).
We write $G=\varprojlim\limits_{N\in \Nc}G/N$.

In this way the 
topological group
$G$ is \emph{approximated} by the Lie groups $G/N$ with $N\in \Nc$.
%\end{notation}
\end{remark}

We now briefly indicate two examples of groups whose representation theory
will be studied in Examples \ref{O4_cor1} and \ref{O4_cor2},
respectively.

\begin{example}
\normalfont
Let $\{G_j\}_{j\ge 1}$ be any sequence of Lie groups.
Their Cartesian product $G:=\prod\limits_{j\ge 1}G_j$ is a topological group, which in general is not a Lie group.
For every $k\ge 0$ define
$$N_k:=\prod\limits_{j\ge k+1}G_j
\simeq \underbrace{\{1\}\times\cdots\times\{1\}}_{k\ \text{\rm times}}\times G_{k+1}\times G_{k+2}\times\cdots\subseteq G.$$
Thus every $N_k$ is a closed normal subgroup of $G$ and $G/N_k\simeq G_1\times\cdots\times G_k$ is a Lie group.
Then it is easily checked that the family $\Nc:=\{N_k\mid k\ge 0\}$ satisfies the above conditions in Remark~\ref{N0}, 
hence the topological group $G$ is the projective limit of the sequence of Lie groups $\{G_1\times\cdots\times G_k\}_{k\ge 1}$.
\end{example}

\begin{example}\label{solenoid}
\normalfont
Let $G$ be any connected Lie group with a sequence of discrete central subgroups
$$G\supseteq \Gamma_1\supseteq\Gamma_2\supseteq\cdots$$
with $\bigcap\limits_{k\ge 1}\Gamma_k=\{\1\}$.
Then we have the projective system of Lie groups
$$G/\Gamma_1\to G/\Gamma_2\to\cdots$$
which generalizes the solenoids discussed in \cite{Cz74}.
\end{example}

\begin{remark}\label{ya}
\normalfont
Every \emph{connected} locally compact group $G$ admits a family $\Nc$ as in Remark~\ref{N0},
for instance the family of \emph{all} its compact normal subgroups $N\subseteq G$ for which the locally compact group $G/N$ is a Lie group.
This fact is known as Yamabe's theorem (see for instance \cite[Lemma 2.3]{BNi15}).
\end{remark}

\section{Factor representations of semigroups and projective limits of groups}\label{factor}

In this section we establish some general results in representation theory of semitopological monoids,
with applications to representations of projective limits,
motivated by the following facts.
Let $G$ be any almost connected locally compact group and denote by $\Nc_0(G)$ its family
of compact normal subgroups $N\subseteq G$ for which $G/N$ is a Lie group.
By \cite[Th. 2.1]{Lip72}, $G=\varprojlim\limits_{N\in \Nc_0(G)}G/N$ implies
$\widehat{G}=\varinjlim\limits_{N\in \Nc_0(G)}\widehat{G_{N}}$, and in particular
$$\widehat{G}=\bigcup\limits_{N\in \Nc_0(G)}\widehat{G_{N}}$$
where $\widehat{G_{N}}$ is the image of $\widehat{G/N}$ in $\widehat{G}$ under the injective map
$\widehat{p_{N}}:\widehat{G/N}\rightarrow \widehat{G}$.
The dual space $\widehat{G}$ of equivalence classes of irreducible unitary representations,
is thus realized as an inductive limit of the dual spaces of Lie groups.
Here we generalize the above result in several directions,
inasmuch as we study factor representations of not necessarily compact semitopological monoids,
rather than irreducible representations of locally compact groups.

For any semigroup $S$ we denote by $\ell^\infty(S)$
the commutative unital $C^*$-algebra of all
complex bounded functions on $S$ with the $\sup$ norm $\Vert\cdot\Vert_\infty$.
For each $t\in S$ we define
$$L_t\colon\ell^\infty(S)\to\ell^\infty(S)\quad\text{and}\quad
R_t\colon\ell^\infty(S)\to\ell^\infty(S)$$
by
$(L_tf)(s)=f(ts)$ and $(R_tf)(s)=f(st)$
whenever $s\in S$ and $f\in\ell^\infty(S)$.

If the semigroup $S$ is equipped with a topology then
we denote by $\Cb(S)$ the set of all continuous functions in $\ell^\infty(S)$.
When $S$ is a semitopological semigroup we denote
$\LUCb(S)$ the set of all {\it left uniformly continuous}
bounded complex functions on $S$.
That is,
$f\in\LUCb(S)$ if and only if
$f\in\Cb(S)$ and the mapping
$S\to\Cb(S)$, $s\mapsto L_sf$,
is continuous.
Similarly, 
we define the set $\RUCb(S)$
of all {\it right uniformly continuous} bounded complex functions on $S$
by the above condition with $L_s$ replaced by $R_s$.
Moreover, 
we introduce the set
$\UCb(S):=\LUCb(S)\cap\RUCb(S)$
consisting of all {\it uniformly continuous} bounded
complex functions on $S$.
It is clear that all of the sets $\LUCb(S)$, $\RUCb(S)$ and $\UCb(S)$
are unital $C^*$-subalgebras of~$\Cb(S)$.

Next denote $\Tc:=\LUCb(S)$.
A {\it mean} on $S$
is a linear functional $\mu\colon{\mathcal T}\to\CC$
satisfying $0\le \mu(\varphi)$ if $0\le\varphi\in\Tc$,
and $\Vert\mu\Vert=\mu(\1)=1$, where $\1$ is the constant function equal to~$1$ on~$G$.
It is well known that this implies $\Re\,\mu(\varphi)=\mu(\Re\,\varphi)$ for every $\varphi\in\Tc$.
Now note that ${\mathcal T}$ is invariant under
the operators $L_t$ for each $t\in S$.
We say that a mean $\mu\colon{\mathcal T}\to\CC$ is
{\it left invariant} if $\mu\circ L_t=\mu$ for all $t\in S$.
If this is the case, then $S$ said to be {\it amenable}.
For instance, every compact group is amenable, and a left invariant mean is given by its probability Haar measure,
suitably normalized.
Also, every solvable topological group is amenable (see \cite{Da57}).

\begin{lemma}\label{am0}
Let $\Hc$ be a complex Hilbert space with its contraction semigroup~$C(\Hc)$
endowed with the strong operator topology.
Let $\alpha\colon T\to S$
and $\pi\colon S\to C(\Hc)$ be continuous morphisms of semitopological semigroups,
and
$$(\forall x,y\in\Hc)\quad \psi^\pi_{x,y}\colon S\to\CC,\
\psi^\pi_{x,y}(\cdot):=(x\mid \pi(\cdot)y).$$
Then one has $\psi^\pi_{x,y}\in \RUCb(S)$ and
$\psi^\pi_{x,y}\circ\alpha=\psi^{\pi\circ\alpha}_{x,y}\in \RUCb(T)$ for all $x,y\in\Hc$.
If moreover $\pi$ is continuous with respect to the strong$^*$ operator topology on~$C(\Hc)$,
then one also has $\psi^\pi_{x,y}\in \LUCb(S)$ and
$\psi^\pi_{x,y}\circ\alpha=\psi^{\pi\circ\alpha}_{x,y}\in \LUCb(T)$ for all $x,y\in\Hc$.
\end{lemma}

\begin{proof}
For all $s\in S$ we have  $L_s(\psi^\pi_{x,y})=\psi^\pi_{\pi(s)^*x,y}$
and
$R_s(\psi^\pi_{x,y})=\psi^\pi_{x,\pi(s)y}$
and then
$$\Vert L_s(\psi^\pi_{x,y})-\psi^\pi_{x,y}\Vert_\infty
=\Vert\psi^\pi_{\pi(s)^*x-x,y}\Vert_\infty
\le\Vert \pi(s)^*x-x\Vert\cdot \Vert y\Vert$$
and similarly
$$\Vert R_s(\psi^\pi_{x,y})-\psi^\pi_{x,y}\Vert_\infty\le \Vert x\Vert\cdot\Vert \pi(s)y-y\Vert.$$
Thus, using the continuity of the representation~$\pi$,
we obtain $\psi^\pi_{x,y}\in \RUCb(S)$.
Then the remaining assertions follow directly.
We only recall that continuity of $\pi$  with respect to the strong$^*$ operator topology on~$C(\Hc)$
means that
$$\lim\limits_{s\to s_0}\Vert\pi(s)x-\pi(s_0)x\Vert
=\lim\limits_{s\to s_0}\Vert \pi(s)^*x-\pi(s_0)^*x\Vert=0$$
for all $x\in\Hc$ and $s_0\in S$.
\end{proof}

We use the following terminology.
A \emph{semigroup representation} is any semigroup morphism $\pi\colon S\to\Bc(\Hc)$, where
$\Hc$ is a complex Hilbert space and $\Bc(\Hc)$ is regarded as  multiplicative semigroup.
We then define
$$\pi(S)':=\{a\in\Bc(\Hc)\mid(\forall s\in S)\ \pi(s)a=a\pi(s)\},$$
and we say that $\pi$ is a \emph{factor representation} if
$\{0\}$ and $\Hc$ are the only closed linear subspaces
that are invariant to all operators from $\pi(S)\cup\pi(S)'$.
It is clear that every unitary irreducible representation of a group is a factor representation.
More generally, if for every $a\in\pi(G)$ we have $a^*\in\pi(G)$,
then $\pi(G)''$ is a von Neumann algebra,
and $\pi$ is a factor representation if and only if $\pi(G)''$ is a factor in the sense of the theory of von Neumann algebras,
that is, the center of $\pi(G)''$ is equal to the set of all scalar multiples of the identity operator on~$\Hc$.

\begin{lemma}\label{am0.5}
Let $\pi\colon S\to\Bc(\Hc)$ be any semigroup representation,
and for any normal subsemigroup $N_0\subseteq S$ define
$\Hc_0:=\{x\in\Hc\mid(\forall n\in N_0)\ \pi(n)x=x\}$.
Then $\Hc_0$ is a linear subspace of $\Hc$ that is invariant to any operator from $\pi(S)\cup\pi(S)'$.
\end{lemma}

\begin{proof}
It is clear that if $a\in\Bc(\Hc)$ and $\pi(s)a=a\pi(s)$ for every $s\in S$, then $a\Hc_0\subseteq\Hc_0$.
Moreover, for arbitrary $s\in S$ and $n\in N_0$ we have
$ns=sn_1$ for some $n_1\in N_0$,
by the definition of the fact that $N_0$ is a normal subsemigroup of $S$.
Then for all $x\in\Hc_0$ we obtain $\pi(n)\pi(s)x=\pi(ns)x=\pi(sn_1)x=\pi(s)\pi(n_1)x=\pi(s)x$,
and thus $\pi(s)\Hc_0\subseteq\Hc_0$.
We have thus proved that $\Hc_0$ is invariant both to $\pi(S)$ and to $\pi(S)'$.
\end{proof}

With Lemmas \ref{am0}--\ref{am0.5} at hand,
we now prove the following generalization of \cite[Satz 1]{Ko82}.

\begin{proposition}\label{am1}
Let $S$ be a semitopological monoid,
and $\Hc$ be a complex Hilbert space with its contraction semigroup~$C(\Hc)$
regarded as a topological semigroup with the strong$^*$ operator topology.
Assume that $\pi\colon S\to C(\Hc)$ is a continuous morphism of monoids,
which is also a factor representation.

Then
there exists a neighbourhood $V$ of $\1\in S$
such that for every amenable semitopological semigroup $N$
and every continuous normal morphism of semigroups $\iota\colon N\to S$ with $\iota(N)\subseteq V$ we have
$N\subseteq \Ker(\pi\circ\iota)$.
\end{proposition}

\begin{proof}
Let $x_0\in\Hc$ be any vector with $\Vert x_0\Vert=1$.
Since $\pi\colon S\to C(\Hc)$ is a morphism of monoids, one has $\pi(\1)=\1$.
Then $(x_0\mid\pi(\1)x_0)=1$, hence there exists a neighbohood~$V$ of $\1\in S$,
depending on $x_0$ and $\pi$ and satisfying
\begin{equation}\label{am1_proof_eq1}
(\forall s\in V)\quad \Re\,(x_0\mid\pi(s)x_0)\ge 1/2.
\end{equation}
Now let $N$ be any amenable semitopological semigroup
and $\iota\colon N\to S$ be any continuous morphism for which $\iota(N)$ is a normal subsemigroup of~$S$
and $\iota(N)\subseteq V$.
We will prove that $(\pi\circ\iota)(n)=\1\in\Bc(\Hc)$ for every $n\in N$.
To this end, we define
the closed linear subspace of $\Hc$,
$$\Hc_N:=\{x\in\Hc\mid \pi(\iota(n))x=x\}$$
and we will prove that $\Hc_N=\Hc$.
It follows by Lemma~\ref{am0.5} applied for $N_0:=\iota(N)$
that $\Hc_N$ is invariant both to $\pi(S)\cup\pi(S)'$.
As $\pi$ is a factor representation, the equality $\Hc_N=\Hc$ will follow
as soon as we will have proved that $\Hc_N\ne\{0\}$.

For arbitrary $x\in\Hc$,
the function $\varphi_x\colon N\to \CC$, $\varphi_x(\cdot):=(x\mid\pi(\iota(\cdot))x_0)$,
satisfies $\varphi_x\in\UCb(N)$ by Lemma~\ref{am0}.
Since the semigroup $N$ is amenable, there exists a linear functional $\mu\colon\LUCb(N)\to\CC$
with $\mu(\1)=1$ and $\mu(L_n(\varphi_x))=\mu(\varphi_x)$ for every $x\in\Hc$ and $n\in N$.
One has $\vert \mu(\varphi_x)\vert\le \sup\limits_N\vert \varphi_x(\cdot)\vert\le\Vert x\Vert$
and $x\mapsto\mu(\varphi_x)$ is a linear functional on $\Hc$,
hence by Riesz' theorem there exists a unique vector $x_1\in\Hc$
with $\mu(\varphi_x)=(x\mid x_1)$ for every $x\in\Hc$.
We then obtain
$$(x\mid\pi(\iota(n))x_1)=(\pi(\iota(n))^*x\mid x_1)
=\mu(\varphi_{\pi(\iota(n))^*x}).$$
It is easily checked that
$\varphi_{\pi(\iota(n))^*x}=L_n(\varphi_x)$
hence, using the property $\mu(L_n(\varphi_x))=\mu(\varphi_x)$ for every $n\in N$,
we obtain by the above equalities
$$(x\mid\pi(\iota(n))x_1)=\mu(\varphi_x)=(x\mid x_1)$$
for all $x\in\Hc$, hence $\pi(\iota(n))x_1=x_1$ for every $n\in N$.
This shows that $x_1\in\Hc_N$.
Note that, by $\iota(N)\subseteq V$ and \eqref{am1_proof_eq1},
we obtain $\Re\,\varphi_{x_0}\ge 1/2$ on $N$, and then, since $\mu\ge 0$ and $\mu(\1)=1$, we have
$\Re(x_0\mid x_1)=\Re(\mu(\varphi_{x_0}))\ge 1/2$.
This shows that $x_1\ne 0$, hence $\Hc_N\ne\{0\}$, and this completes the proof,
as we explained above.
\end{proof}

\begin{theorem}\label{am2a}
Let $S$ be a semitopological monoid with a filter basis $\Nc$ of normal submonoids
converging to the identity.
Assume that every element of $\Nc$ is an amenable semitopological semigroup with its induced topology from~$S$.

Then for every factor representation $\pi\colon G\to C(\Hc)$  which is continous with respect to the strong$^*$ operator topology
there exist $N_0\in\Nc$ satisfying the following equivalent conditions:
\begin{enumerate}[(i)]
\item\label{am2a_item1}
There exists a factor representation $\pi_0\colon G/N_0\to C(\Hc)$
which is continous with respect to the strong$^*$ operator topology and satisfies
$\pi=\pi_0\circ p_{N_0}$.
\item\label{am2a_item2} One has $N\subseteq\Ker\pi$.
\end{enumerate}
\end{theorem}

\begin{proof}
It is clear that conditions \eqref{am2a_item1} and \eqref{am2a_item2} from the statement are equivalent.
Now let $V$ be the neighborhood of $\1\in G$ given by Lemma~\ref{am1}.
Since $\Nc$ is converging to the identity, there exists $N_0\in\Nc$ with $N_0\subseteq V$.
Using Lemma~\ref{am1} for the inclusion map $\iota\colon N_0\hookrightarrow G$ and the fact that $N_0$ is amenable,
it follows that $N_0\subseteq\Ker\pi$, hence there exists a unique unitary representation $\pi_0\colon S/N_0\to\Bc(\Hc)$
with $\pi=\pi_0\circ p_N$.
Since $p_N\colon S\to S/N$ is a quotient map and $\pi$ is continuous, it then follows that $\pi_0$ is continuous
and is a factor representation, and this completes the proof.
\end{proof}

\begin{corollary}\label{am2}
Let $G$ be a topological group with a filter basis $\Nc$ of closed normal subgroups
converging to the identity.
Assume that every element of $\Nc$ is an amenable topological group with its induced topology from~$G$.

Then for every factor representation $\pi\colon G\to U(\Hc)$
which is continuous with respect to the strong operator topology
there exist $N_0\in\Nc$ and a factor representation $\pi_0\colon G/N_0\to\Bc(\Hc)$ with
$\pi=\pi_0\circ p_{N_0}$.
\end{corollary}

\begin{proof}
Use Theorem~\ref{am2a} and the fact if a unitary representation $\pi\colon G\to U(\Hc)$
is continuous with respect to the strong operator topology,
then it is actually continuous with respect to the strong$^*$ operator topology.
\end{proof}

\begin{remark}\label{am_ref}
\normalfont
Corollary~\ref{am2} is a generalization of several results from the earlier literature on unitary representations;
see for instance \cite[Th. 2.1]{Lip72}, \cite[Prop. 2.2]{Mo72}, \cite{Mi75}, \cite[Cor. to Lemma 1]{Mag81}, \cite[Folg. 3--4]{Ko82},
where the elements of the filter basis of normal submonoids $\Nc$
are amenable because they are either compact or solvable subgroups.
\end{remark}

\begin{corollary}\label{pro0}
Let $G$ be any locally compact group with a filter basis $\Nc$ of compact normal subgroups converging to the identity.
For every
irreducible representation $\pi\colon G\to U(\Hc)$
there exist $N_0\in\Nc$ and an irreducible representation $\pi_0\colon G/N_0\to U(\Hc)$  with $\pi=\pi_0\circ p_{N_0}$.
\end{corollary}

\begin{proof}
This is a special case of Corollary~\ref{am2}
\end{proof}

\subsection*{Application to infinite direct products of amenable groups}
We now construct the main class of non-locally-compact groups to which the above results can be applied.
To this end we need the following proposition on amenability of any infinite direct product of amenable topological groups,
for later use in the proof of Corollary~\ref{O4_cor2}.
This is a result that we were not able to locate in the existing literature,
althought it has been long known that
in the category of discrete groups an infinite direct product of amenable groups may not be amenable
(see \cite[page 517]{Da57}).
The point of Proposition~\ref{product} below is that the explanation of that seemingly pathological behavior
is that an infinite direct product of discrete topological spaces is not a discrete topological space.
Amenability does behave well with respect to infinite direct products
as soon as we endow any infinite product of discrete groups with its proper topology,
which is an infinite direct product of discrete topologies.

\begin{proposition}\label{product}
If $\{G_j\}_{j\in J}$ is a family of amenable topological groups,
then their direct product $G:=\prod\limits_{j\in J}G_j$ is also an amenable topological group.
\end{proposition}

For proving Proposition~\ref{product} we need Lemmas \ref{dense} and \ref{proj} below.

\begin{lemma}\label{dense}
Let $\{H_j\}_{j\in J}$ be any family of topological groups.
For every finite subset $F\subseteq J$ we denote
$$H_F:=\prod\limits_{j\in F}G_j
\text{ and }
p_F\colon H\to H_F, (h_j)_{j\in J}\mapsto (h_j)_{j\in F}.$$
We define
$$\Ac:=\bigcup_F\Ac_F$$
where $\Ac_F:=\{\psi\circ p_F\mid \psi\in\LUCb(H_F)\}$ as $F$ runs over all finite subsets of $J$.

Then $\Ac$ is a dense subset of $\LUCb(H)$.
\end{lemma}

\begin{proof}
For every finite subset $F\subseteq J$ we also denote
$$H^{(F)}:=\prod\limits_{j\in J\setminus F}H_j.$$
There exists a canonical isomorphism $H_F\simeq H/H^{(F)}$, and thus $p_F$ can be identified with
the quotient map $p_{H^{(F)}}\colon H\to H/H^{(F)}$.
Moreover, we will need the canonical injective morphism
$$\iota_F\colon H_F\hookrightarrow H$$
defined by
$\iota_F((h_j)_{j\in F})=(g_j)_{j\in J}$, where $g_j:=h_j$ if $j\in F$ and $g_j:=\1\in H_j$ if $j\in J\setminus F$.

Now let $\varphi\in\LUCb(H)$ and $\varepsilon>0$ arbitrary.
Then there exists a neighborhood $V$ of $\1\in H$ with $\vert\varphi(g)-\varphi(h)\vert<\varepsilon$
for all $g,h\in H$ with $gh^{-1}\in V$.
By the definition of the topology of $H$, there exist a finite subset $F\subseteq J$
and for every $j\in F$ there exists a neighborhood $V_j$ of $\1\in H_j$ with
\begin{equation}\label{dense_proof_eq1}
\Bigl(\prod\limits_{j\in F}V_j\Bigr)\times H^{(F)}\subseteq V.
\end{equation}
We will prove that
\begin{equation}\label{dense_proof_eq2}
\Vert \varphi-(\varphi\circ\iota_F)\circ p_F\Vert_\infty\le\varepsilon
\end{equation}
and this will complete the proof, because $\varphi\in\LUCb(H)$ implies $\varphi\circ\iota_F\in\LUCb(H_F)$
(since $\iota_F$ is a homomorphism of topological groups) hence $(\varphi\circ\iota_F)\circ p_F\in\Ac$.
In order to prove \eqref{dense_proof_eq2}, we note that for arbitrary $g=(g_j)_{j\in J}\in H$
we have $(\iota_F\circ p_F)(g)=h:=(h_j)_{j\in J}$,
where $h_j=g_j$ if $j\in F$ and $h_j=\1\in G_j$ if $j\in J\setminus F$.
Hence $gh^{-1}=(k_j)_{j\in J}$, where
$k_j=\1\in V_j$ for $j\in F$ and $k_j\in G_j$ for $j\in J\setminus F$.
Thus, using \eqref{dense_proof_eq1}, we obtain $g((\iota_F\circ p_F)(g))^{-1}\in V$ for every $g\in H$.
According to the way $V$ was selected, we then obtain \eqref{dense_proof_eq2}, and we are done.
\end{proof}

\begin{remark}
\normalfont
It is easily seen that the method of proof of Lemma~\ref{dense} allows us to obtain a more precise conclusion,
namely that for every $\varphi\in\LUCb(G)$ one has
$$\lim\limits_F\Vert \varphi-(\varphi\circ\iota_F\circ p_F)\Vert_\infty=0$$
where the limit is taken with respect to the upwards directed set of all finite subsets of $J$,
partially ordered with respect to inclusion.
\end{remark}

\begin{lemma}\label{proj}
Let $H$ be a topological group with a filter basis $\Nc$ of closed normal subgroups.
For every $N\in\Nc$ we assume the quotient group $H/N$ is amenable,
we denote by $p_N\colon H\to H/N$ the corresponding quotient map,
and we define $\Ac_N:=\{\psi\circ p_N\mid \psi\in\LUCb(H/N)\}$.
Also define $\Ac:=\bigcup\limits_{N\in\Nc}\Ac_N$  with its closure in $\LUCb(H)$ denoted by $\overline{\Ac}$.
Then $\overline{\Ac}$ is invariant to the left-translation operator $L_s\colon \LUCb(H)\to \LUCb(H)$ for every $s\in H$,
and there exists a left invariant mean on $\overline{\Ac}$.
\end{lemma}

\begin{proof}
Using that $p_N\colon H\to H/N$ is a surjective morphism of groups,
it folows that $\Ac_N\subseteq\LUCb(H)$ is a closed linear subspace
that contains the constant functions and is invariant to the operator $L_s\colon \LUCb(H)\to\LUCb(H)$
for arbitrary $s\in H$.
In fact, for all $x\in H$ we have
$$(L_s(\psi\circ p_N))(x)
=(\psi\circ p_N)(sx)
=\psi(p_N(s)p_N(x))
=(L_{p_N(s)}\psi)(p_N(x)) $$
hence $L_s(\psi\circ p_N)=(L_{p_N(s)}\psi)\circ p_N$,
and this shows that $\Ac_N$ is invariant under $L_s$.

We denote by $\Mg_N$ the set of all  linear functionals $\mu\colon\LUCb(H)\to\CC$
satisfying $0\le \mu(\varphi)$ if $0\le\varphi\in\Ac_N$;
$\Vert\mu\Vert=\mu(\1)=1$, where $\1$ is the constant function equal to~$1$ on~$H$;
and $\mu\circ L_s=\mu$ on $\Ac_N$ for all $s\in H$.
It is clear that $\Mg_N$ is a weak$^*$-compact subset of the unit ball of $\LUCb(H)^*$.
To see that $\Mg_N\ne\emptyset$, we may select any left invariant mean $\mu_0\colon \LUCb(H/N)\to\CC$,
since $G/N$ is amenable.
Then we define $\widetilde{\mu}_0\colon\Ac_N\to\CC$, $\widetilde{\mu}_0(\psi\circ p_N):=\mu_0(\psi)$,
and we use the Hahn-Banach theorem to find a linear functional $\mu\colon\LUCb(H)\to\CC$
with $\mu\vert_{\Ac_N}=\widetilde{\mu}_0$ and $\Vert\mu\Vert=\Vert\widetilde{\mu}_0\Vert=1$.
Then it is easily checked that $\mu\in\Mg_N$, hence $\Mg_N\ne\emptyset$.

We now recall that if $N_1,N_2\in\Nc$ with $N_1\supseteq N_2$,
then we have a natural surjective morphism of groups
$p_{N_2,N_1}\colon G/N_2\to G/N_1$ satisfying $p_{N_1}=p_{N_2,N_1}\circ p_{N_2}$,
and this directly implies that $\Ac_{N_1}\subseteq\Ac_{N_2}$ and then $\Mg_{N_1}\supseteq\Mg_{N_2}$.
Thus $\{\Mg_N\}_{N\in\Nc}$ is a family of weak$^*$-compact nonempty subsets of the unit ball of $\LUCb(H)^*$
and the intersection of every finite subfamily is nonempty.
Since the unit ball of $\LUCb(H)^*$ is weak$^*$-compact by the Banach-Alaoglu theorem,
it then follows that there exists $\mu\in\bigcap\limits_{N\in\Nc}\Mg_N$.
Then for every $s\in S$ and $N\in\Nc$ we have $\mu\circ L_s=\mu$ on $\Ac_N$
hence, we obtain $\mu\circ L_s=\mu$ on $\Ac$.
Thus $\mu$ is a left invariant mean on $\overline{\Ac}$, and this completes the proof.
\end{proof}

\begin{proof}[Proof of Proposition~\ref{product}]
We first introduce notation similar to that used in the proofs of Lemmas \ref{proj} and \ref{dense}.
Thus, for every finite subset $F\subseteq J$ we denote
$$G_F:=\prod\limits_{j\in F}G_j,\ G^{(F)}:=\prod\limits_{j\in J\setminus F}G_j,
\ p_F\colon G\to G_F, (g_j)_{j\in J}\mapsto (g_j)_{j\in F}.$$
One has the canonical isomorphism $G_F\simeq G/G^{(F)}$, and thus $p_F$ can be identified with
the quotient map $p_{G^{(F)}}\colon G\to G/G^{(F)}$.
Moreover, there is the canonical injective morphism
$$\iota_F\colon G_F\hookrightarrow G$$
defined by
$\iota_F((h_j)_{j\in F})=(g_j)_{j\in J}$, where $g_j:=h_j$ if $j\in F$ and $g_j:=\1\in G_j$ if $j\in J\setminus F$.
Using these maps we define
$$\Ac:=\bigcup_F\Ac_F$$
where $\Ac_F:=\{\psi\circ p_F\mid \psi\in\LUCb(G_F)\}$ as $F$ runs over all finite subsets of $J$.

We now apply Lemma~\ref{proj} for $H:=G$
and
$$\Nc:=\{G^{(F)}\mid F\subseteq J,\ \vert F\vert<\infty\}.$$
For every finite subset $F\subseteq J$ the group
$G/G^{(F)}\simeq G_F$ is a direct product of finitely many amenable groups, hence it is well known that $G_F$ is amenable.
Thus the hypotheses of Lemma~\ref{proj} are fulfilled, and we obtain
a left invariant mean on $\overline{\Ac}$.
On the other hand, one has $\overline{\Ac}=\LUCb(G)$ by Lemma~\ref{dense}.
It follows that there exists a left invariant mean on $G$, hence $G$ is amenable, and we are done.
\end{proof}

\begin{remark}
\normalfont
The proof of Proposition~\ref{product} uses some ideas from the proof of \cite[Lemma 1]{Ko82},
which seems however to need some extra arguments,
because it is not clear how to generalize the above Lemma~\ref{dense} from infinite direct products to arbitrary projective limits.
In this connection, it would be interesting to see if every projective limit of amenable topological groups is
again an amenable topological group with respect to its projective limit topology.
Just as in the special case of infinite direct products of discrete groups,
it has long been known that a projective limit of amenable discrete groups may not be amenable as a discrete group
(see \cite[page 517]{Da57} again).
\end{remark}

Now we can draw the following consequence of Corollary~\ref{am2}.

\begin{corollary}\label{am3}
Let $\{G_j\}_{j\in J}$ be any family of amenable topological groups,
with their direct product $G:=\prod\limits_{j\in J}G_j$.
For every factor representation $\pi\colon G\to U(\Hc)$
which is continuous with respect to the strong operator topology
there exists a finite subset $F\subseteq J$  and a factor representation $\pi_0\colon G_F\to\Bc(\Hc)$ with
$\pi=\pi_0\circ p_F$, where $G_F:=\prod\limits_{j\in F}G_j$ and $p_F\colon G\to G_F$ is the natural projection.
\end{corollary}

\begin{proof}
Use Corollary~\ref{am2} along with Proposition~\ref{product}.
\end{proof}

\section{Coadjoint orbits and representations}\label{main}

In this section we study coadjoint orbits of pro-Lie groups and their corresponding unitary representations.
The first part of the following definition is suggested by a similar, more general, notion
of topological group with Lie algebra (see \cite[Def. 2.11]{HM07}).
We choose to work here only with topological groups whose Lie algebras are locally convex
because this class is general enough for our purposes and, on the other hand,
the local convexity assumption ensures that the dual spaces of these Lie algebras are nontrivial, hence
it is meaningful to study their corresponding coadjoint actions.

\begin{definition}\label{grlcvx}
\normalfont
Let $G$ be any topological group and endow its set of continuous 1-parameter subgroups,
$$\Lie(G):=\{X\in\Cc(\RR,G)\mid
(\forall t,s\in\RR)\quad X(t+s)=X(t)X(s)\}.$$
with the topology of uniform convergence on the compact subsets of~$\RR$.

We say that $G$ is a \emph{topological group with locally convex Lie algebra} if
the topological space $\Lie(G)$ has the structure of a locally convex Lie algebra over~$\RR$,
whose scalar multiplication, vector addition and bracket
satisfy the following conditions for all
$t,s\in\RR$ and $X_1,X_2\in\Lie(G)$:
\begin{eqnarray}
\nonumber
 (t\cdot X_1)(s) &=& X_1(ts);  \\
\nonumber
(X_1+X_2)(t) &=& \lim\limits_{n\to\infty}(X_1(t/n)X_2(t/n))^n;\\
\nonumber
[X_1,X_2](t^2) &=& \lim\limits_{n\to\infty}(X_1(t/n)X_2(t/n)X_1(-t/n)X_2(-t/n))^{n^2},
\end{eqnarray}
where the convergence is assumed to be uniform on the compact subsets of~$\RR$.
If this is the case, then we define
$$\Lie(G)^*:=\{\xi\colon\Lie(G)\to\RR\mid\xi\text{ is linear and continuous}\}$$
and we regard $\Lie(G)^*$ as a locally convex real vector space, endowed with its weak dual topology,
that is, the topology of pointwise convergence on~$\Lie(G)$.

We recall (see for instance \cite{HM07} or \cite{BNi15}
%\cite[Def. 2.1]{BNi16}
)
the \emph{adjoint action}
$$\Ad_G\colon G\times\Lie(G)\to\Lie(G),\quad (g,X)\mapsto\Ad_G(g)X:=g X(\cdot)g^{-1} $$
and this defines by duality the \emph{coadjoint action}
$$\Ad_G^*\colon G\times\Lie(G)^*\to\Lie(G)^*,\quad (g,\xi)\mapsto\Ad_G^*(g)\xi:=\xi\circ \Ad_G(g^{-1}). $$
We denote by $\Lie(G)^*/G$ the set of all coadjoint orbits, that is, the orbits of the above coadjoint action.
\end{definition}

\begin{example}
\normalfont
Every connected locally compact group is a topological group with locally convex Lie algebra
(see for instance \cite{Gl57} or \cite{La57}).
\end{example}

\begin{lemma}\label{coadj_lemma}
In Definition~\ref{grlcvx}, the coadjoint action is well defined.
\end{lemma}

\begin{proof}
We must prove that if $\xi\in\Lie(G)^*$ and $g\in G$ then $\Ad_G^*(g)\xi\in\Lie(G)^*$,
that is, the function $\Ad_G^*(g)\xi\colon\Lie(G)\to\RR$ is indeed continuous and linear.
Since $\Ad_G(g)$ is continuous (see \cite[Prop. 2.28]{HM07}), the function $\Ad^*_G(g)\xi$ is continous on~$\Lie(G)$.
To prove that $\Ad_G(g)$ is linear,
we first note that for every $t\in\RR$ the evaluation map
$\Lie(G)\to G$, $X\mapsto X(t)$, is continuous.
Then for all $X, Y\in\Lie(G)$ and $t\in\RR$ we may write
%\begin{array}{lll}
\begin{align}
\bigl(\Ad_G(g)(X+Y)\bigr)(t)
&= \Ad_G(g)((X+Y)(t)) \nonumber\\
&=\Ad_G(g)(\lim\limits_{n\rightarrow \infty}(X(t/n)Y(t/n))^n) \nonumber\\
&=\lim\limits_{n\to \infty}\left(\Ad_G(g)X(t/n)\Ad_G(g)Y(t/n)\right)^n \nonumber\\
&=\lim\limits_{n\to \infty}\left(\frac{1}{n}\Ad_G(g)X(t).\frac{1}{n}\Ad_G(g)Y(t)\right)^n \nonumber\\
&=\lim\limits_{n\to \infty}\left(\frac{1}{n}(\Ad_G(g)X).\frac{1}{n}\Ad_G(g)Y\right)^n(t) \nonumber\\
&=(\Ad_G(g)X + \Ad_G(g)Y)(t)  \nonumber\\
&=\Ad_G(g)X(t) + \Ad_G(g)Y(t)  \nonumber
%\end{array}
\end{align}
Moreover, for all $s,t\in\RR$ and $X\in \Lie(G)$, one has
$$\Ad_G(g)(s X)(t)=\Ad_G(g)(X)(s t)=(s \Ad_G(g)(X))(t).$$
Thus $\Ad_G(g)$ is linear and for all $\xi\in \Lie(G)^*$, the function $\Ad^*_G(g)\xi=\xi\circ \Ad_G(g^{-1})$ is linear.
This completes the proof.
\end{proof}

The result of the above Lemma~\ref{coadj_lemma} should be regarded in the context of linearity properties
of differentials on topological groups (see for instance \cite{BB11} and the references therein).

%The following lemma is inspired by the proof of \cite[Lemma 2]{Koe82}.

\begin{lemma}\label{O1}
Let $p\colon G_1\to G_2$ be a continuous surjective morphism of topological groups with Lie algebras.
If $\Oc_2\in\Lie(G_2)^*/G_2$ is a coadjoint orbit of $G_2$,
then the set $\Oc_1:=\Lie(p)^*(\Oc_2)\subseteq\Lie(G_1)^*$ is a coadjoint orbit of $G_1$.
\end{lemma}

\begin{proof}
One has
\begin{equation}\label{O1_proof_eq1}
\Oc_1=\{\xi\circ\Lie(p)\mid\xi\in\Oc_2\}.
\end{equation}
Let us fix $\xi_2\in\Oc_2$ and denote $\xi_1:=\xi_2\circ\Lie(p)\in\Oc_1$.
We must prove that the coadjoint action of $G_1$ on $\Oc_1$ is transitive,
hence that for arbitrary $\eta_1\in\Oc_1$ there exists $g_1\in G_1$ with $\xi_1\circ\Ad_{G_1}(g_1)=\eta_1$.

Since $\eta_1\in\Oc_1$, it follows by \eqref{O1_proof_eq1} that there exists $\eta_2\in\Oc_2$ with $\eta_2\circ\Lie(p)=\eta_1$.
But $\Oc_2$ is a coadjoint orbit of $G_2$ and $\xi_2,\eta_2\in\Oc_2$, hence there exists $g_2\in G_2$
with $\eta_2=\xi_2\circ\Ad_{G_2}(g_2)$.

As $p\colon G_1\to G_2$ is surjective, there exists $g_1\in G_1$ with $p(g_1)=g_2$, and then
$\eta_2=\xi_2\circ\Ad_{G_2}(p(g_1))$, which further implies
\begin{equation}\label{O1_proof_eq2}
\eta_2\circ\Lie(p)=\xi_2\circ\Ad_{G_2}(p(g_1))\circ\Lie(p)\colon \Lie(G_1)\to\RR.
\end{equation}
On the other hand, every $X\in\Lie(G_1)$ is a function $X\colon\RR\to G_1$ and we have
$(\Lie(p))(X)\in\Lie(G_2)$ defined by $(\Lie(p))(X):=p\circ X$,
hence
$$\begin{aligned}
(\Ad_{G_2}(p(g_1))\circ\Lie(p))(X)
& =p(g_1)((p\circ X)(\cdot))p(g_1)^{-1} \\
& =p\circ (g_1X(\cdot)g_1^{-1}) \\
& =\Lie(p)((\Ad_{G_1}(g_1))(X)).
\end{aligned}$$
We thus obtain the commutative diagram
$$\begin{CD}
\Lie(G_1) @>{\Ad_{G_1}(g_1)}>> \Lie(G_1) \\
@V{\Lie(p)}VV @VV{\Lie(p)}V \\
\Lie(G_2) @>{\Ad_{G_2}(p(g_1))}>> \Lie(G_2)
\end{CD}$$
for $g_1\in G_1$.
Then~\eqref{O1_proof_eq2} is equivalent to
$\eta_2\circ\Lie(p)=\xi_2\circ\Lie(p)\circ\Ad_{G_1}(g_1)$, hence $\eta_1=\xi_1\circ\Ad_{G_1}(g_1)$,
and this completes the proof.
\end{proof}

\begin{lemma}\label{O2}
Let $G$ be a topological group with a filter basis $\Nc$ of closed normal subgroups
converging to the identity.
Assume that for every $N\in\Nc$ the quotient $G/N$ is a Lie group.

%In the setting of Notation~\ref{N0}, we have the following facts:
Then the following assertions hold:
\begin{enumerate}[(i)]
\item For every $N\in\Nc$ the linear map
$\Lie(p_N)^*\colon \Lie(G/N)^*\to\Lie(G)^*$ is injective and moreover
$$\Lie(G)^*=\bigcup_{N\in\Nc}\Lie(p_N)^*(\Lie(G/N)^*).$$
\item A set $\Oc\subseteq\Lie(G)^*$ is a coadjoint $G$-orbit if and only if there exist $N\in\Nc$
and a coadjoint $(G/N)$-orbit $\Oc_0\subseteq\Lie(G/N)^*$ with $\Lie(p_N)^*(\Oc_0)=\Oc$.
\end{enumerate}
\end{lemma}

\begin{proof}
For every $N\in\Nc$ the map $\Lie(p_N)\colon\Lie(G)\to\Lie(G/N)$ is surjective by
%\cite[Th. 1]{Gl57}
\cite[Lemma 4.19]{HM07},
hence the dual map $\Lie(p_N)^*\colon \Lie(G/N)^*\to\Lie(G)^*$ is injective.
Similarly, if $N_1,N_2\in \Nc$ with $N_1\subseteq N_2$, then the map
$\Lie(p_{N_1,N_2})\colon \Lie(G/N_1)\to \Lie(G/N_2)$ is surjective,
and we thus obtain the projective limit of finite-dimensional Lie algebras $\{\Lie(G/N)\}_{N\in\Nc}$
whose projective limit is the locally convex Lie algebra $\Lie(G)$.
It follows by \cite[\S 22, no. 6, Th. (6)]{Koe69} that $\Lie(G)^*$ is the inductive limit of the inductive system of finite-dimensional vector spaces $\{\Lie(G/N)^*\}_{N\in\Nc}$ connected by the linear maps $\Lie(p_{N_1,N_2})^*\colon \Lie(G/N_2)^*\to \Lie(G/N_1)^*$
for all $N_1,N_2\in \Nc$ with $N_1\subseteq N_2$,
with the canonical maps $\Lie(p_N)^*\colon\Lie(G/N)^*\to\Lie(G)^*$ for all $N\in\Nc$.
This completes the proof of the first assertion.

For the second assertion, it follows by Lemma~\ref{O1} that for every
$N\in\Nc$
and a coadjoint $(G/N)$-orbit $\Oc_0\subseteq\Lie(G/N)^*$, the set $\Lie(p_N)^*(\Oc_0)\subseteq\Lie(G)^*$ is a coadjoint $G$-orbit.
Conversely, let $\Oc\subseteq\Lie(G)^*$ be any coadjoint $G$-orbit.
Select any $\xi\in\Oc$.
Then it follows by the preceding assertion that there exists
$N\in\Nc$ with $\xi\in\Lie(p_N)^*(\Lie(G/N)^*)$, that is, there exists $\xi_0\in\Lie(G/N)^*$
with $\xi=\Lie(p_N)^*(\xi_0)=\xi_0\circ\Lie(p_N)$.
Let $\Oc_0\subseteq\Lie(G/N)^*$ be the coadjoint orbit of $\xi_0$.
It follows by Lemma~\ref{O1} again that the set $\Lie(p_N)^*(\Oc_0)\subseteq\Lie(G)^*$ is a coadjoint $G$-orbit.
But we have $\xi=\Lie(p_N)^*(\xi_0)\in\Lie(p_N)^*(\Oc_0)\cap\Oc$, hence $\Oc=\Lie(p_N)^*(\Oc_0)$,
because two coadjoint orbits are either disjoint from each other, or equal.
This completes the proof.
\end{proof}

We are now in a position to obtain one of the main results of the present paper,
which extends Kirillov's correspondence beyond the category of Lie groups.
Here we use the notation $\Lie(\cdot)^*_{\ZZ}$ for the set of all integral linear functionals
on the dual of the Lie algebra of any connected nilpotent Lie group.
(See Lemma~\ref{O3a} below.)

\begin{theorem}\label{O4}
Let $G$ be a complete topological group with a filter basis $\Nc$ of closed normal subgroups
converging to the identity.
Assume that for every $N\in\Nc$ the quotient $G/N$ is a connected nilpotent Lie group and the topological group $N$ is amenable.

Then there exists a well-defined bijective correspondence
$$\Psi_G\colon \widehat{G}\to \Lie(G)^*/G,\quad [\pi]\mapsto\Oc^\pi$$
between the equivalence classes of unitary irreducible representations of $G$ and the set of all coadjoint $G$-orbits contained in the $G$-invariant set
$$\Lie(G)^*_{\ZZ}:=\{\xi\in\Lie(G)^*\mid (\exists N\in\Nc)(\exists \eta\in\Lie(G/N)^*_\ZZ)\quad \xi=\eta\circ\Lie(p_N)\}.$$
Via the aforementioned correspondence,
every unitary irreducible representation $\pi\colon G\to\Bc(\Hc)$ is associated to the coadjoint $G$-orbit $\Oc^\pi:=\Lie(p_N)^*(\Oc_0)\subseteq\Lie(G)^*_{\ZZ}$, where $N\in\Nc$ and $\Oc_0\subseteq \Lie(G/N)^*_{\ZZ}$
is the coadjoint $(G/N)$-orbit associated with a unitary irreducible representation $\pi_0\colon G/N\to\Bc(\Hc)$
satisfying $\pi_0\circ p_N=\pi$.
\end{theorem}

\begin{proof}
We recall from Remark~\ref{N0} that
$G=\varprojlim\limits_{N\in \Nc}G/N$, which implies  $\Lie(G)=\varprojlim\limits_{N\in \Nc}\Lie(G/N)$.
Then
$$
\Lie(G)^*_\ZZ=\varinjlim\limits_{N\in \Nc}\Lie(G/N)^*_\ZZ
$$
is realized as an inductive limit of the corresponding sets $\Lie(G/N)^*_{\ZZ}$ where $G/N$ is a connected nilpotent Lie group, and in particular
$$
\Lie(G)^*_\ZZ=\bigcup\limits_{N\in \Nc}\Lie(p_N)^*\Lie(G/N)^*_{\ZZ}
$$
where $\Lie(p_{N})^*:\Lie(G/N)^*_{\ZZ}\rightarrow \Lie(G)^*_\ZZ$ is the dual map of the canonical projection $p_N:G\rightarrow G/N$.
This amounts to the following:
For any coadjoint $G$-orbit $\Oc\subseteq\Lie(G)^*_\ZZ$, there exist $N_0\in\Nc$ and a coadjoint $G/N_0$-orbit
$\Oc_0\subseteq \Lie(G/N_0)^*_{\ZZ}$ such that $\Oc=\Lie(p_{N_0})^*(\Oc_0)$.

Given $\pi$ as in the statement, we use Corollary~\ref{am2}
in order to find $N\in\Nc$ and
$\pi_0\colon G/N\to\Bc(\Hc)$ with $\pi_0\circ p_N=\pi$.
Then, using Lemma~\ref{O3a} for the connected nilpotent Lie group $G/N$,
we find the coadjoint orbit $\Oc_0\subseteq \Lie(G/N)^*_{\ZZ}$ associated with~$\pi_0$.
Now $\Oc^\pi:=\Lie(p_N)^*(\Oc_0)\subseteq \Lie(G)^*_{\ZZ}$ is a coadjoint $G$-orbit by Lemma~\ref{O1}.

We now prove that the coadjoint orbit $\Oc^\pi$ does not depend on the choice of
$N\in\Nc$ and $\pi_0\colon G/N\to\Bc(\Hc)$ with $\pi_0\circ p_N=\pi$, or, equivalently, $N\subseteq\Ker\pi$
(see Theorem~\ref{am2a}). 
To this end, let $N_1,N_2\in\Nc$, $N_1\cup N_2\subseteq\Ker\pi$,
and $\pi_j\colon G/N_j\to\Bc(\Hc)$ with $\pi_j\circ p_{N_j}=\pi$,
and let $\Oc_j\subseteq \Lie(G/{N_j})^*_{\ZZ}$ be the coadjoint orbit corresponding to $\pi_j$ for $j=1,2$. 
We claim that $\Lie(p_{N_1})^*(\Oc_1)=\Lie(p_{N_2})^*(\Oc_2)$. 

To prove that claim, first suppose that $N_1\subseteq N_2$.
Then one has the map $p_{N_1, N_2}:G/N_1\rightarrow G/N_2$, $gN_1\mapsto gN_2$, 
which is a
continuous surjective homomorphism of connected nilpotent Lie groups,
hence $\Lie(p_{N_1,N_2})^*:\Lie(G/N_2)^*_{\ZZ}\rightarrow \Lie(G/N_1)^*_{\ZZ}$ is injective.
Since $p_{N_1, N_2}\circ p_{N_1}=p_{N_2}$, it follows that 
$\pi_2\circ p_{N_1, N_2}=\pi_1$ by the construction from the proof of Theorem~\ref{am2a}. 
Then, by Proposition~\ref{O3b}, we obtain $\Lie(p_{N_1,N_2})^*(\Oc_2)=\Oc_1$.
Using again the equality $p_{N_1,N_2}\circ p_{N_1}=p_{N_2}$,
one has $\Lie(p_{N_1})^*\circ\Lie(p_{N_1,N_2})^*=\Lie(p_{N_2})^*$, hence
$$
\Lie(p_{N_1})^*(\Oc_1)=\Lie(p_{N_1})^*\Lie(p_{N_1, N_2})^*(\Oc_2)=\Lie(p_{N_2})^*(\Oc_2)
$$
as claimed above.

We thus proved that the correspondence $\Psi_G$ from the statement is well defined. 
This correspondence is surjective.
In fact, let $\Oc$ be any coadjoint $G$-orbit.
By Lemma~\ref{O2} there exist $N\in\Nc$ and a coadjoint $G/N$-orbit $\Oc_0$ with $L(p_N)^*(O_0)=\Oc$.
Then, using  Lemma~\ref{O3a} for the nilpotent Lie group $G/N$, we find the unitary representation $\pi_0$ associated to $\Oc_0$.
Now, $\pi=\pi_0\circ p_N$ is a unitary representation of $G$ associated with the coadjoint $G$-orbit $\Oc$.

For injectivity, let $\Oc$ and $\Oc'$ be two coadjoint $G$-orbits. 
By Lemma~\ref{O2} (and its proof), 
there exist $N\in\Nc$ and two $(G/N)$-orbits $\Oc_0$ and $\Oc'_0$ with $\Oc=L(p_N)^*(\Oc_0)$ and $\Oc'=L(p_N)^*(\Oc'_0)$.
Since the map $L(p_N)^*$ is injective, 
one has $\Oc\ne\Oc'$ if and only if $\Oc_0\neq \Oc'_0$, and by Lemma~\ref{O3a}, 
this is further equivalent to the fact that $\Oc_0$ and $\Oc_0'$ are associated respectively with 
inequivalent unitary representations $\pi_0$ and $\pi'_0$, 
and that inequivalence is the same thing with inequivalence of the representations hence $\pi=\pi_0\circ p_N$ and $\pi'_0\circ p_N=\pi'$.
This completes the proof.
\end{proof}

A version of Theorem~\ref{O4} in which the quotients $G/N$ are simply connected is stated in \cite[Satz 2]{Ko82},
however, the corresponding proof seems to be incomplete.

In the setting of Theorem~\ref{O4}, the bijective correspondence from
the equivalence classes of unitary irreducible representations of $G$ to the coadjoint orbits of $G$
can be regarded as a generalized Kirillov correspondence for~$G$.
We also note that, as a direct consequence of Theorem~\ref{O4}, the $G$-invariant set
$\Lie(G)^*_{\ZZ}$ of ``integral functionals'' is actually independent on the choice of
the filter basis $\Nc$ satisfying the conditions from the statement of the theorem.

\begin{corollary}\label{O4_cor1}
If $G$ is a connected locally compact nilpotent group,
then there are a $G$-invariant subset $\Lie(G)^*_{\ZZ}$ and a bijective correspondence
$\Psi_G\colon \widehat{G}\to \Lie(G)^*/G$
onto the set of all coadjoint $G$-orbits contained in $\Lie(G)^*_{\ZZ}$.
Moreover, if $\Nc$ is a filter basis of closed normal subgroups of $G$ converging to  the identity
for which $G/N$ is a Lie group for every $N\in\Nc$, then
$$\Lie(G)^*_{\ZZ}:=\{\xi\in\Lie(G)^*\mid (\exists N\in\Nc)(\exists \eta\in\Lie(G/N)^*_\ZZ)\quad \xi=\eta\circ\Lie(p_N)\}.$$
\end{corollary}

\begin{proof}
Using Remark~\ref{ya}, there exists a filter basis $\Nc$ of closed normal subgroups of $G$ converging to the identity
for which $G/N$ is a Lie group if $N\in\Nc$.
Moreover, since $G$ is a nilpotent group and is connected,
it follows that for every $N\in\Nc$ the quotient $G/N$ is a connected nilpotent Lie group.
On the other hand, $N$ is also nilpotent, hence amenable.
Thus the assertion follows by Theorem~\ref{O4}.
\end{proof}

\begin{example}\label{sol}
\normalfont
Let us revisit Example~\ref{solenoid} from the present perspective.
Assume that $\gg$ is a finite-dimensional nilpotent Lie algebra
with its corresponding Lie group $\widetilde{G}=(\gg,\cdot)$ defined by the Baker-Campbell-Hausdorff multiplication.
Then $\widetilde{G}$ is a connected, simply connected, nilpotent Lie group, whose exponential map
is the identity map of~$\gg$.
Denote by $\zg$ the center of $\gg$, and let $n:=\dim\zg$,
so that we may assume $\zg=\RR^n$.
Then $\widetilde{Z}=(\zg,+)$ is the center of $\widetilde{G}$,
and let us define the lattice $\Gamma_1:=\ZZ^n\subset\RR^n=\zg$.
Now select any sequence of discrete subgroups
$(\ZZ^n,+)=\Gamma_1\supseteq\Gamma_2\supseteq\cdots$
with $\bigcap\limits_{k\ge 1}\Gamma_k=\{0\}$, and with
the corresponding projective system of Lie groups
$\widetilde{G}/\Gamma_1\to \widetilde{G}/\Gamma_2\to\cdots$.
Define
$$G:=\varprojlim_{k\ge 1}\widetilde{G}/\Gamma_k\hookrightarrow\prod_{k\ge 1}\widetilde{G}/\Gamma_k.$$
It follows by \cite[Lemma and Th. 2]{Cz74} that $G$ is a locally compact group
and one has a canonical isomorphism of Lie algebras
$$\gg\simeq\Lie(G).$$
Moreover, $G$ is nilpotent because it is isomorphic to a closed subgroup of
the topological group $\prod\limits_{k\ge 1}\widetilde{G}/\Gamma_k$,
which is easily seen to be nilpotent.
Consequently, by Corollary~\ref{O4_cor1}, there exists a bijection
$\Psi_G\colon \widehat{G}\to \Lie(G)^*/G$
onto
the set of all coadjoint $G$-orbits contained in the set $\Lie(G)^*_{\ZZ}$ of ``integral functionals'' on~$\Lie(G)$.

We now provide a precise description of $\Lie(G)^*_{\ZZ}$.
For $k=1,2,\dots$ we have $\Lie(\widetilde{G}/\Gamma_k)=\gg$,
hence, by Lemma~\ref{O3a},
$\Lie(\widetilde{G}/\Gamma_k)^*_{\ZZ}:=\{\xi\in\gg^*\mid \xi(\Gamma_k)\subseteq \ZZ\}$.
Thus, by Corollary~\ref{O4_cor1},
$$\Lie(G)^*_{\ZZ}:=\{\xi\in\gg^*\mid (\exists k\ge 1)\quad \xi(\Gamma_k)\subseteq \ZZ\}.$$
\end{example}

\begin{corollary}\label{O4_cor2}
Let $\{G_j\}_{j\in J}$ be any family of connected nilpotent Lie groups,
with their direct product topological group $G:=\prod\limits_{j\in J}G_j$.
Define
$$\Lie(G)^*_{\ZZ}:=\{\xi\in\Lie(G)^*\mid (\exists F\in\Fc)(\exists\eta\in\Lie(G_F)^*_{\ZZ})\quad \xi=\eta\circ\Lie(p_F)\}$$
where $\Fc$ is the set of all finite subsets $F\subseteq J$,
and for every $F\in\Fc$ we define $G_F:=\prod\limits_{j\in F}G_j$ and $p_F\colon G\to G_F$ is the natural projection.
Then $\Lie(G)^*_{\ZZ}$ is a $G$-invariant subset of $\Lie(G)$ and there is a bijective correspondence
$\Psi_G\colon \widehat{G}\to \Lie(G)^*/G$ 
onto the set of all coadjoint $G$-orbits contained in $\Lie(G)^*_{\ZZ}$.
\end{corollary}

\begin{proof}
For every $F\in\Fc$ the group
$G^{(F)}:=\prod\limits_{j\in J\setminus F}G_j$ is amenable by Proposition~\ref{product}
and can be naturally regarded as a closed normal subgroup of $G$.
Moreover,
one has the canonical isomorphism $G_F\simeq G/G^{(F)}$,
and $p_F$ can be identified with
the quotient map $p_{G^{(F)}}\colon G\to G/G^{(F)}$.
Since the group $G_F$ is a connected nilpotent Lie group, we may apply Theorem~\ref{O4},
and we are done.
\end{proof}

\begin{example}
\normalfont
Here we illustrate Corollary~\ref{O4_cor2} by the simplest example, which already shows that
the usual $C^*$-algebraic approach to group representation theory breaks down for infinite direct products
of non-compact locally compact groups, hence the description of their unitary duals
in terms of coadjoint orbits is particularly important whenever it is available,
as it is the only description known so far.

Specifically, let $J:=\NN=\{0,1,2,\dots\}$ and $G_j=(\RR,+)$ for every $j\in J$.
Then $G=(\RR^\NN,+)$, $\Lie(G)=\RR^\NN$ is the vector space of all sequences of real numbers,
and $\Lie(G)^*=\RR^{(\NN)}$ is the vector space of all finitely supported sequences of real numbers,
and the coadjoint orbits of $G$ can be identified with the points of $\Lie(G)^*$ since $G$ is an abelian group.
It thus follows by Corollary~\ref{O4_cor2} or by direct verification that there exists
a bijection $\Psi_G\colon\widehat{G}\to\RR^{(\NN)}$.
However, as the vector space $\RR^{(\NN)}$ is infinite dimensional, it is not locally compact,
hence it is not homeomorphic to the spectrum of any $C^*$-algebra.
Consequently, the irreducible representation theory of $G$ can be described via coadjoint orbits
but not via any $C^*$-algebra.

See however \cite{Gr05}, \cite{GN13}, and the references therein
for an interesting $C^*$-algebraic approach to representation theory of topological groups
that are not locally compact.
\end{example}

\appendix

\section{Complements on representations of nilpotent Lie groups}
\label{appA}

Our basic references for representation theory of nilpotent Lie groups are 
\cite{Ki62} and \cite{CG90}.
In this section we record a few results that are needed in the main body of our paper. 
For the reader's convenience we provide self-contained proofs for some of the results 
that we found more difficult to locate in the literature. 

\begin{lemma}\label{O3}
Let $p\colon G_1\to G_2$ be a continuous surjective morphism of connected and simply connected nilpotent Lie groups.
Let $\pi_2\colon G_2\to\Bc(\Hc)$ be any unitary irreducible representation of $G_2$,
associated with a coadjoint orbit $\Oc_2\subseteq\Lie(G_2)^*$. 
Then $\pi_1:=\pi_2\circ p \colon G_1\to\Bc(\Hc)$is a unitary irreducible representation of $G_1$,
and its corresponding coadjoint orbit is  $\Oc_1:=\Lie(p)^*(\Oc_2)\subseteq\Lie(G_1)^*$.
\end{lemma}

\begin{proof}
It follows by Lemma~\ref{O1} that $\Oc_1$ is a coadjoint orbit of~$G_1$.

Let $\ell_2\in\Oc_2$ and $\ell_1=\Lie (p)^*(\ell_2)$. 
To construct $\pi_2$, we select a real polarization $\hg_2$ in $\ell_2$, i.e $\hg_2\subset \Lie(G_2)$ such that:
\begin{itemize}
\item[1.] $\hg_2$ is a subalgebra of $\Lie(G_2)$,
\item[2.] $\hg_2$ is subordinated to $\ell_2$, i.e $\langle\ell_2,[\hg_2,\hg_2]\rangle=0$,
\item[3.] $\hg_2$ is a maximal isotropic subspace, i.e.,
$$
\dim\hg_2=\frac{1}{2}\dim\Lie(G_2)+\frac{1}{2}\dim\Lie(G_2)(\ell_2).
$$
\end{itemize}

We denote by $H_2=\exp\hg_2$ the analytic subgroup with Lie algebra $\hg_2$. 
The exponential of the character $i\ell_2$ of $\hg_2$ is a character $\chi_{\ell_2}$ of $H_2$, defined by:
$$
\chi_{\ell_2}(\exp~X)=e^{i\langle \ell_2,X\rangle}\qquad (X\in\hg_2).
$$
The representation $\pi_2$ is an induced unitary representation:
$$
\pi_2=\text{Ind}_{H_2}^{G_2}~\chi_{\ell_2}.
$$
It is realized in the completion of the space $\Ci(G_2:H_2)$ of $C^\infty$ functions $\varphi$ on $G_2$, with compact support modulo $H_2$, 
satisfying 
$$
\varphi(xh)=\chi_{\ell_2}(h)^{-1}\varphi(x)\quad\text{ for }x\in G_2,\ h\in H_2,\quad\text{and}\quad
\Vert \varphi\Vert^2=\int_{G_2/H_2}|\varphi(\dot{x})|^2~d\dot{x}
$$
The representation $\pi_2$ is defined by:
$$
\left(\pi_2(x)\varphi\right)(y)=\varphi(x^{-1}y),\quad
\text{ for }\varphi\in\Ci(G_2:H_2).$$
First, the Lie algebra of $G_2=p(G_1)$ is $\Lie(G_2)= \Lie(p)(\Lie(G_1))$, hence $\Lie(p)$ is surjective.
Moreover
$$\aligned
\Lie(G_1)(\ell_1)
:=&\{X\in \Lie(G_1)\mid \ad^*(X)\ell_1=0\} \\
=&\{X\in \Lie(G_1)\mid (\forall Y\in\Lie(G_1))\ \langle (\Lie p)^*(\ell_2),[X,Y]\rangle=0\}\\
=&\{X\in \Lie(G_1)\mid (\forall Y\in\Lie(G_1))\ \langle \ell_2,[\Lie(p)(X),\Lie(p)(Y)]\rangle=0\}\\
=&\{X\in\Lie(G_1)\mid \Lie(p)(X)\in\Lie(G_2)(\ell_2)\}\\
=&(\Lie(p))^{-1}(\Lie(G_2)(\ell_2)).
\endaligned
$$
This implies that $\Ker\Lie (p)\subset\Lie (G_1)(\ell_1)$,
$\Lie(p)|_{\Lie(G_1)(\ell_1)}\colon\Lie(G_1)(\ell_1) \rightarrow \Lie (G_2)(\ell_2)$ is surjective, and
$$
\dim\Lie(G_2)(\ell_2)+\dim(\Ker \Lie(p))=\dim(\Lie(G_1)(\ell_1)).
$$
Then, we put $\hg_1:=(\Lie(p))^{-1}(\hg_2)$.
Since $\hg_2$ is a maximal isotropic subspace, one has $\Lie(G_2)(\ell_2)\subset\hg_2$, hence
$$
\Ker \Lie(p)\subset\Lie(G_1)(\ell_1)=(\Lie(p))^{-1}\Lie(G_2)(\ell_2)\subset(\Lie(p))^{-1}\hg_2=\hg_1.
$$
The same argument gives
$$
\dim \hg_2+\dim\Ker \Lie(p)=\dim\hg_1, \quad\dim\Lie(G_2)+\dim\Ker \Lie(p)=\dim\Lie(G_1).
$$
Now $\hg_1$ is a polarization in $\ell_1$. In fact:
\begin{itemize}
\item[1.] $\hg_1$ is a subalgebra of $\Lie(G_1)$: For all $X$, $Y\in\hg_1$,
$$
\Lie(p)([X,Y])=[\Lie(p)(X),\Lie(p)(Y)]\in[\hg_2,\hg_2]\subset\hg_2,
$$
then $[X,Y]\in\hg_1$.
\item[2.] $\hg_1$ is subordinated to $\ell_1$:
$$\aligned
\langle\ell_1,[\hg_1,\hg_1]\rangle&=\langle \Lie(p)^*\ell_2,[\hg_1,\hg_1]\rangle=\langle \ell_2,\Lie(p)([\hg_1,\hg_1])\rangle\\
&=\langle \ell_2,[\Lie(p)(\hg_1),\Lie(p)(\hg_1)]\rangle\subset\langle \ell_2,[\hg_2,\hg_2]\rangle=0.
\endaligned
$$
\item[3.] $\hg_1$ is  a maximal isotropic subspace:
$$\aligned
\dim\hg_1&=\dim\hg_2+\dim\Ker \Lie(p)\\
&=\frac{1}{2}\big(\dim\Lie(G_2)+\dim \Ker \Lie(p)+\dim\Lie(G_2)(\ell_2)+\dim\Ker \Lie(p)\big)\\
&=\frac{1}{2}\big(\dim\Lie(G_1)+\dim\Lie(G_1)(\ell_1)\big).
\endaligned
$$
\end{itemize}

Then the representation associated to the coadjoint orbit $\Oc_1$ is (with the same notation as above):
$$
\rho_1=\text{Ind}_{H_1}^{G_1}~\chi_{\ell_1}.
$$
Note that as $\Ker \Lie(p)\subset \hg_1$, one has $\Ker p=\exp(\Ker \Lie(p))\subset \exp(\hg_1)=H_1$.
Hence, $\forall\psi\in\Ci(G_1:H_1)$, $x\in G_1$ and $k\in\Ker p$, if $k=\exp X$ then
$$
\begin{aligned}
\psi(xk)=\chi_{\ell_1}(k)^{-1}\psi(x)
&=\ee^{-\ie\langle \ell_1,X\rangle}\psi(x)
= \ee^{-\ie\langle \Lie(p)^* \ell_2,X\rangle}\psi(x) \\
&= \ee^{-\ie\langle \ell_2,\Lie(p)(X)\rangle}\psi(x)
=\psi(x).
\end{aligned}
$$
The function $\psi$ passes to the quotient, i.e.,
there exists $\varphi\colon G_2=G_1/\Ker p\to\CC$ such that $\varphi\circ p=\psi$.
Moreover, if $h=\exp X\in H_2$, then $X\in\hg_2=\Lie(p)(\hg_1)$, $X=\Lie(p)(Y)$, for each $x=p(y)$ in $G_2$, we have:
$$
\varphi(xh)=\varphi(p(y)p(\exp Y))=\psi(y\exp Y)=\ee^{-\ie\langle \ell_1,Y\rangle}\psi(y)
=\ee^{-\ie\langle \ell_2,X\rangle}\varphi(x).
$$
In other words, $\varphi\in\Ci(G_2:H_2)$.
Conversely, if $\varphi\in\Ci(G_2:H_2)$,
then $\varphi\circ p\in\Ci(G_1:H_1)$.
Thus, there is a linear bijection $\Ci(G_2:H_2)\to\Ci(G_1:H_1)$, 
$\varphi\mapsto \varphi\circ p$.

On the other hand, we have:
$$
G_2/H_2\simeq(G_1/\Ker p)/(H_1/\Ker p)\simeq G_1/H_1,
$$
and
$$
\Vert\psi\Vert^2=\int_{G_1/H_1}|\psi(\dot{x})|^2~d\dot{x}=\int_{G_2/H_2}|\varphi(\dot{x})|^2~d\dot{x}=\Vert \varphi\Vert^2.
$$
The bijection above extends uniquely to a unitary %transformation $U\varphi$
operator $\Hc_{\pi_2}\to\Hc_{\rho_1}$.

Note that, if $x,~y\in G_1$
$$\aligned
\big(\rho_1(x)U\varphi\big)(y)
&=(U\varphi)(x^{-1}y)=(\varphi\circ p)(x^{-1}y)
=\varphi(p(x)^{-1}p(y)) \\
&=\Big((\pi_2\circ p)(x)(\varphi\circ p)\Big)(y)
=\Big((\pi_2\circ p)(x)(U\varphi)\Big)(y),
\endaligned
$$
that is,
$\rho_1(x)=(\pi_2\circ p)(x)=\pi_1(x)$, $\rho_1=\pi_1$,
and this completes the proof.
\end{proof}

\begin{lemma}\label{O3a}
Let $G$ be any connected and simply connected nilpotent Lie group with a discrete normal subgroup $\Gamma\subseteq G$,
and denote by $p\colon G\to G/\Gamma$ the quotient map.
Denote by $\log_G\colon\gg\to G$ the inverse of the exponential map of~$G$.
Then for every unitary irreducible representation $\pi\colon G/\Gamma\to \Bc(\Hc)$ the representation
$\pi\circ p \colon G\to \Bc(\Hc)$ is irreducible, and its corresponding coadjoint orbit $\Oc\subseteq\gg^*$
has the property that for all $\xi\in\Oc$ one has $\xi(\log_G(\Gamma))\subseteq \ZZ$.
One thus obtains a bijective correspondence from $\widehat{G/\Gamma}$
onto the set of all coadjoint $G$-orbits contained in the $G$-invariant set
$\gg^*_{\ZZ}:=\{\xi\in\gg^*\mid \xi(\log_G(\Gamma))\subseteq \ZZ\}$.
\end{lemma}

\begin{proof}
See \cite[Th. 8.1]{Ki62} and \cite[Th. 4.4.2]{CG90}.
\end{proof}

Now we can prove the following generalization of Lemma~\ref{O3}
to connected nilpotent Lie groups that may not be simply connected.

\begin{proposition}\label{O3b}
Let $p\colon G_1\to G_2$ be a continuous surjective morphism of connected nilpotent Lie groups.
Let $\pi_2\colon G_2\to\Bc(\Hc)$ be any unitary irreducible representation of $G_2$,
associated with a coadjoint orbit $\Oc_2\subseteq\Lie(G_2)^*$. 
Then $\pi_1:=\pi_2\circ p \colon G_1\to\Bc(\Hc)$is a unitary irreducible representation of $G_1$,
and its corresponding coadjoint orbit is  $\Oc_1:=\Lie(p)^*(\Oc_2)\subseteq\Lie(G_1)^*$.
\end{proposition}

\begin{proof}
It follows by Lemma~\ref{O1} that $\Oc_1$ is a coadjoint orbit of~$G_1$.
Let $\widetilde{G}_j$ be the universal covering group of $G_j$,
with a suitable discrete central subgroup $\Gamma_j$ and the quotient map
$p_j\colon\widetilde{G}_j\to \widetilde{G}_j/\Gamma_j=G_j$.
Since $\widetilde{G}_1$ is simply connected,
it follows that the Lie group morphism $p\circ p_1\colon \widetilde{G}_1\to G_2$
lifts to a unique Lie group morphism $\widetilde{p}\colon\widetilde{G}_1\to\widetilde{G}_2$
for which the diagram
\begin{equation}\label{03b_proof_eq0}
\begin{CD}
\widetilde{G}_1 @>{\widetilde{p}}>> \widetilde{G}_2 \\
@V{p_1}VV @VV{p_2}V \\
G_1 @>{p}>> G_2
\end{CD}
\end{equation}
is commutative,
hence $p_2\circ \widetilde{p}=p\circ p_1$.
Since $\Gamma_j=\Ker p_j$, it then follows that $\widetilde{p}(\Gamma_1)=\Gamma_2$.
Now the conclusion follows, using of Lemmas \ref{O3} and \ref{O3a}.

Specifically, since $p\colon G_1\to G_2$ is surjective and $\pi_2$ is irreducible,
it is easily checked that $\pi_2\circ p$ is an irreducible representation of~$G_1$.
In order to identify the coadjoint orbit of $G_1$ associated with $\pi_2\circ p$ via Lemma~\ref{O3a},
we first note that
$$\pi_2\circ p\circ p_1=(\pi_2\circ p_2)\circ\widetilde{p}$$
by \eqref{03b_proof_eq0}.
The above equality implies by Lemma~\ref{O3} that
if we denote by $\widetilde{\Oc}_2\subseteq\Lie(\widetilde{G}_2)^*=\Lie(G_2)^*$
the coadjoint orbit of $\widetilde{G}_2$ that is associated with $\pi_2\circ p_2$,
then
\begin{equation}\label{O3b_proof_eq1}
\widetilde{\Oc}_1:=\Lie(\widetilde{p})^*(\Oc_2)\subseteq\Lie(\widetilde{G}_1)^*
\end{equation}
is the coadjoint orbit of~$\widetilde{G}_1$ that is associated with $\pi_2\circ p\circ p_1$,
where we have used the notation $\Oc_2$ introduced in the statement.
On the other hand, for $j=1,2$, the group morphism
$p_j\colon\widetilde{G}_j\to G_j$ is a covering map,
hence $\Lie(p_j)\colon\Lie(\widetilde{G}_j)\to \Lie(G_j)$ is an isomorphism of Lie algebras.
By Lemma~\ref{O3a}, the coadjoint orbit of $G_2$ associated with $\pi_2$
is just the coadjoint orbit of $\widetilde{G}_2$ associated with $\pi_2\circ p_2$.
More precisely, using the vector space isomorphism $\Lie(p_2)^*\colon \Lie(G_2)^*\to \Lie(\widetilde{G}_2)^*$,
one has
\begin{equation}\label{O3b_proof_eq2}
\Lie(p_2)^*(\Oc_2)=\widetilde{\Oc}_2.
\end{equation}
Similarly, by Lemma~\ref{O3a} again, the coadjoint orbit of $G_1$ associated with $\pi_2\circ p$
is just the coadjoint orbit of $\widetilde{G}_1$ associated with $(\pi_2\circ p)\circ p_1$.
More precisely, using the vector space isomorphism $\Lie(p_1)^*\colon \Lie(G_1)^*\to \Lie(\widetilde{G}_1)^*$,
one has
\begin{equation}\label{O3b_proof_eq3}
\Lie(p_1)^*(\Oc_1)=\widetilde{\Oc}_1.
\end{equation}
Also, by \eqref{03b_proof_eq0}, one has $p\circ p_1=p_2\circ\widetilde{p}$,
hence $\Lie(p)\circ \Lie(p_1)=\Lie(p_2)\circ \Lie(\widetilde{p})$,
and then
$\Lie(p_1)^*\circ \Lie(p)^*=\Lie(\widetilde{p})^*\circ \Lie(p_2)^*$,
which further implies by \eqref{O3b_proof_eq2}, \eqref{O3b_proof_eq1}, and \eqref{O3b_proof_eq3},
$$\Lie(p_1)^*(\Lie(p)^*(\Oc_2))=\Lie(\widetilde{p})^*(\Lie(p_2)^*(\Oc_2))
=\Lie(\widetilde{p})^*(\widetilde{\Oc}_2)=\widetilde{\Oc}_1=\Lie(p_1)^*(\Oc_1).$$
Now, as  $\Lie(p_1)^*$ is a vector space isomorphism, we obtain $\Lie(p)^*(\Oc_2)=\Oc_1$, as claimed,
and this completes the proof.
\end{proof}


\begin{thebibliography}{1000000}

%%%%%%%%%%%%%%%%%%%%%%%%
%%%%%%%%%%%%%%%%%%%%%

\bibitem[ASZ11]{ASZ11}
D.~Arnal, M.~Selmi, A.~Zergane,
Universal overgroup.
{\it J. Geom. Phys.} {\bf 61} (2011), no.~1, 217--229.

\bibitem[Ba91]{Ba91}
W.~Banaszczyk,
{\it Additive subgroups of topological vector spaces}.
Lecture Notes in Mathematics, 1466. Springer-Verlag, Berlin, 1991.

\bibitem[BB11]{BB11}
I.~Belti\c t\u a, D.~Belti\c t\u a,
On differentiability of vectors in Lie group representations.
{\it J. Lie Theory} {\bf 21} (2011), no.~4, 771--785.

\bibitem[BNi15]{BNi15}
D.~Belti\c t\u a, M.~Nicolae,
Moment convexity of solvable locally compact groups.
{\it J. Lie Theory} {\bf 25} (2015), no.~3, 733--751.

\bibitem[CG90]{CG90}
L.J.~Corwin, F.P.~Greenleaf,
{\it Representations of nilpotent Lie groups and their applications}. Part I.
%Basic theory and examples.
Cambridge Studies in Advanced Mathematics, 18. Cambridge University Press, Cambridge, 1990.

\bibitem[Cz74]{Cz74}
G.~Czichowski,
The structure of connected $LP$-groups with finite-dimensional Lie algebra.
{\it Math. Nachr.} {\bf 62} (1974), 77--81.

\bibitem[Da57]{Da57}
M.M.~Day,
Amenable semigroups.
{\it Illinois J. Math.} {\bf 1} (1957), 509--544.

\bibitem[Gl57]{Gl57}
V.M.~Glushkov,
Lie algebras of locally bicompact groups. (Russian)
{\it Uspehi Mat. Nauk (N.S.)} {\bf 12} (1957) no. 2 (74), 137--142.

\bibitem[Gr05]{Gr05}
H.~Grundling,
Generalising group algebras.
{\it J. London Math. Soc. (2)} {\bf 72} (2005), no. 3, 742--762.

\bibitem[GN13]{GN13}
H.~Grundling, K.-H.~Neeb,
Infinite tensor products of $C_0(\RR)$: towards a group algebra for $\RR^{(\NN)}$.
{\it J. Operator Theory} {\bf 70} (2013), no. 2, 311--353.

\bibitem[HM07]{HM07}
K.H.~Hofmann, S.A.~ Morris,
{\it The Lie theory of connected pro-Lie groups}.
EMS Tracts in Mathematics, 2. European Mathematical Society (EMS), Z\"urich, 2007.

\bibitem[HN09]{HN09}
K.H. Hofmann, K.-H.~Neeb,
Pro-Lie groups which are infinite-dimensional Lie groups.
{\it Math. Proc. Cambridge Philos. Soc.} {\it 146} (2009), no. 2, 351--378.

\bibitem[Ki62]{Ki62}
A.A.~Kirillov,
Unitary representations of nilpotent Lie groups.
{\it Uspehi Mat. Nauk} {\bf 17} (1962), no. 4 (106), 57--110.

\bibitem[Ko82]{Ko82}
M. Koebe,
\"Uber Darstellungen nilpotenter und auflösbarer LP-Gruppen.
{\it Math. Nachr.} {\bf 105} (1982), 193--197.

\bibitem[K\"{o}69]{Koe69}
G.~K\"othe,
{\it Topological vector spaces}. I.
Die Grundlehren der mathematischen Wissenschaften, Band 159.
Springer-Verlag New York Inc., New York, 1969.

\bibitem[La57]{La57}
R.K.~Lashof,
Lie algebras of locally compact groups.
{\it Pacific J. Math.} {\bf 7} (1957), 1145--1162.

\bibitem[Lip72]{Lip72}
R.L.~Lipsman,
Representation theory of almost connected groups.
{\it Pacific J. Math.} {\bf 42} (1972), 453--467.

\bibitem[Mag81]{Mag81}
L.~Magnin,
Some remarks about $C^{\infty }$ vectors in representations of connected locally compact groups.
\textit{Pacific J. Math.} {\bf 95} (1981), 391--400.

\bibitem[Mi75]{Mi75}
D.~Mili\v{c}i\'c,
Representations of almost connected groups.
{\it Proc. Amer. Math. Soc.} {\bf 47} (1975), 517--518.

\bibitem[Mo72]{Mo72}
C.C. Moore, Groups with finite dimensional irreducible representations.
\textit{Trans. Amer. Math. Soc.} {\bf 166}  (1972), 401--410.

\bibitem[Mr77]{Mr77}
S.A.~Morris,
{\it Pontryagin duality and the structure of locally compact abelian groups}.
London Mathematical Society Lecture Note Series, No. 29.
Cambridge University Press, Cambridge-New York-Melbourne, 1977.

\bibitem[Ze11]{Ze11}
A.~Zergane,
{\it S\'eparation des repr\'esentations des groupes de
Lie par des ensembles moments}.
Th\`ese de doctorat, Universit\'es de Monastir et de Dijon, 2011.

\end{thebibliography}
\end{document}